\documentclass{article}

\usepackage[T1]{fontenc}
\usepackage{graphicx}
\usepackage{amssymb , amsmath, amsthm}
\usepackage{dsfont}
\newtheorem{thm}{Theorem}[section]
\newtheorem{prop}[thm]{Proposition}
\newtheorem{lemma}[thm]{Lemma}
\newtheorem{dfn}{Definition}[section]
\newtheorem{rmk}{Remark}[section]
\def\Ker{\mathop{\rm Ker}\nolimits}

\title{Transductive versions of the LASSO\\and the Dantzig Selector}
\author{Pierre Alquier and Mohamed Hebiri}
\date{}

\allowdisplaybreaks

\begin{document}

\maketitle

\begin{abstract} 
We consider the linear regression problem, where the number $p$ of covariates is possibly
larger than the number $n$ of observations $(x_{i},y_{i})_{i\leq i \leq n}$, under sparsity
assumptions. On the one hand, several methods have been successfully proposed to perform this
task, for example the LASSO in \cite{Tibshirani-LASSO} or the Dantzig Selector
in \cite{Dantzig}.
On the other hand, consider new values $(x_{i})_{n+1\leq i \leq m}$. If one wants to estimate the
corresponding $y_{i}$'s, one should think of a specific estimator devoted to this task,
referred in \cite{Vapnik} as a "transductive" estimator.
This estimator may differ from
an estimator designed to the more general task "estimate on the whole domain".
In this work, we propose a generalized version both of the LASSO and the Dantzig Selector,
based on the geometrical
remarks about the LASSO in \cite{CSEL,L1MOH}. The "usual" LASSO and Dantzig Selector, as well as new estimators
interpreted as transductive versions of
the LASSO, appear
as special cases.
These estimators are interesting at least from a theoretical point of view: we can give theoretical
guarantees for these estimators under hypotheses that are relaxed versions of the hypotheses required in the papers
about the "usual" LASSO.
These estimators can also be efficiently computed, with results comparable to the ones of the LASSO.
\end{abstract}

\tableofcontents 

\section{Introduction}

In many modern applications, a statistician often have to deal with very large datasets.
Regression problems may involve a large number of covariates $p$, possibly larger than the
sample size $n$. In this situation, a major issue is dimension reduction, which can be
performed through the selection of a small amount of relevant covariates.
For this purpose, numerous regression methods have been proposed in the literature, ranging
from the classical information criteria such as
$\mathop{\rm AIC}$ \cite{aic} and $\mathop{\rm BIC}$ \cite{bic} to the
more recent sparse methods, known as the LASSO \cite{Tibshirani-LASSO},
and the Dantzig Selector \cite{Dantzig}. Regularized
regression methods have recently witnessed several developments due to the attractive feature
of computational feasibility, even for high dimensional data  (i.e., when the number of
covariates $p$ is large).
We focus on the usual
linear regression model:
\begin{equation} \label{eq_depart}
y_{i}= x_{i} \beta^* + \varepsilon_{i}, \quad \quad i=1,\ldots,n,
\end{equation}
where the design $x_{i}=(x_{i,1},\ldots,x_{i,p}) \in \mathbb{R}^p$ is deterministic,
$\beta^*=(\beta^*_1,\ldots,\beta^*_p)' \in \mathbb{R}^p$ is the unknown parameter
and $\varepsilon_1,\ldots,\varepsilon_n$ are i.i.d. centered Gaussian random
variables with known variance $\sigma^{2}$.
Let $X$ denote the matrix with $i$-th line equal to $x_i$, and let $X_{j}$ denote its $j$-th
column, with $i\in\{1,\ldots,n\}$ and $j\in\{1,\ldots,p\}$. So: $$X=(x'_1 , \ldots, x'_n )'
= (X_1,\ldots,X_p).$$
For the sake of simplicity, we will assume that the observations are normalized in such
a way that $X_{j}'X_{j}/n=1$. 
We denote by $Y$ the vector $Y=(y_{1},\ldots,y_{n})'$.\\
\noindent For all $\alpha \leq 1$ and any vector $v\in\mathds{R}^{d}$, we set $ \|\cdot \|_{\alpha}$,
the norm: $ \|v\|_{\alpha} = (|v_{1}|^{\alpha}+\ldots+|v_{d}|^{\alpha})^{1/\alpha}$.
In particular $\|\cdot\|_{2}$ is the euclidean norm.
Moreover for all $d\in\mathds{N}$, we use the notation
$ \|v\|_{0} = \sum_{i=1}^{d} \mathds{1}(v_{i}\neq 0). $

The problem of estimating the regression parameter in the high dimensional setting have been extensively studied in the statistical literature. Among others, the LASSO \cite{Tibshirani-LASSO} (denoted by $\hat{\beta}^{L}$), the Dantzig Selector \cite{Dantzig}
(denoted by $\hat{\beta}^{DS}$) and the non-negative garrote (in Yuan and Lin \cite{garrotte}, denoted by
$\hat{\beta}^{NNG}$) have been proposed to deal with this problem for a large $p$, even for $p>n$.
These estimators give very good practical results. For instance in~\cite{Tibshirani-LASSO}, simulations
and tests on real data have been provided for the LASSO.
We also refer to \cite{KoltchDant,Koltchl1plus,VanGpLass,VandeGeerSparseLasso,ArnakTsyb,ChriMo7GpLass} for
related work with different estimators: non-quadratic loss, penalties slightly different
from $\ell_1$ and random design.

\noindent From a theoretical point of view, Sparsity Inequalities (SI) have been proved for these
estimators under different assumptions. That is upper bounds of order of $\mathcal{O}\left(\sigma^{2} \|\beta^*\|_{0} \log(p)/n\right)$ for the errors
$(1/n)\|X\hat{\beta} - X\beta^{*}\|_{2}^{2}$
and $\|\hat{\beta} - \beta^{*}\|_{2}^{2}$ have been derived, where $\hat{\beta}$ is one of the estimators mentioned above.
In particular these bounds involve the number of non-zero coordinates in $\beta^{*}$ (multiplied by $\log(p)$),
instead of dimension $p$. Such bounds garanty that under some assumptions, $X \hat{\beta}$ and $\hat{\beta}$ are good estimators of $X\beta^{*}$ and $\beta^{*}$ respectively. According to the LASSO $\hat{\beta}^{L}$, these SI are given for example in \cite{BTWAggSOI,Lasso3}, whereas \cite{Dantzig,Lasso3} provided SI for the Dantzig Selector $\hat{\beta}^{DS}$.
On the other hand, Bunea \cite{Bunea_consist} establishes conditions which ensure $\hat{\beta}^{L}$ and $\beta^{*}$ have the same null coordinates. Analog results for $\hat{\beta}^{DS}$ can be found in~\cite{KarimNormSup}.

Now, let us assume that
we are given additional observations $x_{i}\in\mathbb{R}^p$ for $n+1\leq i \leq m$ (with $m>n$),
and introduce the matrix $Z=(x'_{1} , \ldots, x'_{m} )'$. Assume that
the objective of the statistician is precisely to estimate $Z\beta^{*}$: namely, he cares about predicting what would
be the labels attached to the additional $x_{i}$'s. It is argued in \cite{Vapnik}
that in such a case, a specific estimator devoted to this task should be considered: the transductive estimator.
This estimator differs from an estimator
tailored for the estimation of $\beta^{*}$ or $X\beta^{*}$ like the LASSO. Indeed one usually 
builds an estimator $\hat{\beta}(X,Y)$ and then computes $Z\hat{\beta}(X,Y)$ to estimate
$Z\beta^*$.
The approach taken here is to consider estimators  $\hat{\beta}(X,Y,Z)$ exploiting the knowledge of $Z$, and then to compute $Z\hat{\beta}(X,Y,Z)$.

Some methods in supervised classification or regression were successfully extended to the
transductive setting, such as the well-known Support Vector Machines (SVM) in
\cite{Vapnik}, the Gibbs estimators in \cite{manuscrit}. It
is argued in the semi-supervised learning literature (see for example \cite{semi-sup} for
a recent survey) that taking into account the information
on the design given by the new additional $x_{i}$'s has a stabilizing effect on
the estimator.

In this paper, we study a family of estimators
which generalizes the LASSO and the Dantzig Selector.
The considered family depends on a $q\times p$ matrix $A$, with $q\in\mathds{N}$,
whose choice allows to adapt the estimator to the objective of the statistician.
The choice of the matrix $A$ allows to cover transductive setting.

The rest of paper is organized as follows.
In the next section, we motivate the use of the studied family of estimators through geometrical considerations stated in \cite{L1MOH}.
In Sections \ref{easy} and \ref{hard}, we establish Sparsity Inequalities for these estimators.
A discussion on the assumptions needed to prove the SI is also provided.
In particular, it is shown that the estimators devoted to the transductive setting satisfy these SI with weaker assumptions that those needed by the LASSO or the Dantzig Selector, when $m > p > n$. That is, when the number of news points is large enough.
The implementation of our estimators and some numerical experiments are the purpose of Section \ref{simu}.
The results clearly show that the use of a transductive version of the LASSO may improve the performance
of the estimation. All proofs of the theoretical results are postponed to Section \ref{proofs}.

\section{Preliminaries}

In this section we state geometrical considerations (projections on a confidence region) for the LASSO and the Dantzig Selector. These motivate the introduction of our estimators. Finally we discuss the different objectives considered in this paper.

Let us remind that a definition of the LASSO estimate is given by
\begin{equation}
\label{LASSO-def1}
 \arg\min_{\beta \in \mathds{R}^{p}} \left\{\left\|Y-X\beta\right\|_{2}^{2} + 2\lambda \|\beta\|_{1} \right\}.
\end{equation}
A dual form (in \cite{DualLasso}) of this program is also of interest:
\begin{equation}
\label{LASSO-def2}
\left\{
\begin{array}{l}
\arg\min_{\beta \in \mathds{R}^{p}} \left\|X\beta\right\|_{2}^{2}
\\
\\
s. t. \left\|X'(Y-X\beta)\right\|_{\infty} \leq \lambda;
\end{array}
\right.
\end{equation}
actually it is proved in \cite{CSEL} that any solution of Program \ref{LASSO-def2} is a solution of
Program \ref{LASSO-def1} and that the set $\{X\beta\}$ is the same where $\beta$ is taken among all the
solutions of Program \ref{LASSO-def1} or among all the solutions of \ref{LASSO-def2}. So both programs
are equivalent in terms of estimating $X\beta^{*}$.

Now, let us remind the definition of the Dantzig Selector:
\begin{equation}
\label{DS-def}
\left\{
\begin{array}{l}
\arg\min_{\beta \in \mathds{R}^{p}} \left\|\beta\right\|_{1}
\\
\\
s. t. \left\|X'(Y-X\beta)\right\|_{\infty} \leq \lambda.
\end{array}
\right.
\end{equation}

Alquier \cite{CSEL} observed that both Programs \ref{LASSO-def2} and \ref{DS-def} can be seen as
a projection of the null vector $\mathbf{0_p}$ onto the region $\{\beta:\left\|X'(Y-X\beta)\right\|_{\infty} \leq \lambda\}$ that can
be interpreted as a confidence region, with confidence $1-\eta$, for a given $\lambda$
that depends on $\eta$ (see Lemma \ref{usefullemma} here for example).
The difference between the two programs is the distance (or semi-distance) used for the projection.

Based on these geometrical considerations, we proposed in \cite{L1MOH} to study the following transductive
estimator:
\begin{equation}
\label{ancienTL}
\left\{
\begin{array}{l}
\arg\min_{\beta \in \mathds{R}^{p}} \left\|Z\beta\right\|_{2}^{2}
\\
\\
s. t. \left\|X'(Y-X\beta)\right\|_{\infty} \leq \lambda;
\end{array}
\right.
\end{equation}
that is a projection on the same confidence region, but using a distance adapted to the transductive estimation
problem. We proved a Sparsity Inequality for this estimator exploiting a novel sparsity measure.

In this paper, we propose a generalized version of the LASSO and of the Dantzig Selector,
based on the same geometrical remark.
More precisely for $q\in \mathds{N}^*$, let $A$ be a $q\times p$ matrix. We propose two general estimators,
$\hat{\beta}_{A,\lambda}$ (extension of the LASSO, based on a generalization of Program
\ref{LASSO-def1}) and $\tilde{\beta}_{A,\lambda}$
(transductive Dantzig Selector, generalization of Program \ref{DS-def}). These novel estimators depend on
two tuning parameters: $\lambda>0$ is a regularization parameter, it plays the same role
as the tuning parameter involved
in the LASSO, and the matrix $A$ that will allow to adapt the estimator
to the objective of the statistician. More particularly, depending on the choice of
the matrix $A$, this estimator can be adapted to one of the following objectives:
\begin{itemize}
\item {\bf denoising objective:} the estimation of $X\beta^{*}$, that is a denoised
version of $Y$. For this purpose, we consider the estimator $\hat{\beta}_{A,\lambda}$, with $A=X$. In this case, the estimator will actually be
equal to the LASSO $\hat{\beta}^{L}$ and $\tilde{\beta}_{A,\lambda}$, with the same choice $A=X$ will be equal
to the Dantzig Selector;
\item {\bf transductive objective:} the estimation of $Z\beta^{*}$, by $\hat{\beta}_{A,\lambda}$ or $\tilde{\beta}_{A,\lambda}$, with $A=\sqrt{n/m}Z$.
We will refer the corresponding estimators as the "Transductive LASSO" and "Transductive Dantzig
Selector";
\item {\bf estimation objective:} the estimation of $\beta^{*}$ itself, by $\hat{\beta}_{A,\lambda}$, with $A=\sqrt{n}I$.
In this case,
it appears that both estimators are well defined only in the case $p<n$ and are equal to
a soft-thresholded version of the usual least-square estimator.
\end{itemize}

For both estimators and all the above objectives, we prove SI (Sparsity Inequalities). Moreover, we show that
these estimators can easily be computed.

\section{The "easy case": $ \Ker(X)=\Ker(Z)$}
\label{easy}

In this section, we deal with the "easy case", where $\Ker(A)=\Ker(X)$ (think of
$A=X$, $A=\sqrt{n}I$ or $A=\sqrt{n/m}Z$). This setting
is natural at least in the case $p<n$ where both kernels are equal to $\{0\}$ in general. We provide SI (Sparsity Inequality, Theorem \ref{mainthm}) for the studied estimators, based on the techniques developed in \cite{Lasso3}.

\subsection{Definition of the estimators}

\begin{dfn}
\label{est}
For a given parameter $\lambda\geq 0$ and any matrix $A$ such that $Ker(A)=Ker(X)$,
we consider the estimator given by
$$
 \hat{\beta}_{A,\lambda} \in\arg\min_{\beta \in \mathds{R}^{p}}
\Bigl\{ -2Y'X(\widetilde{X'X})^{-1} (A'A) \beta
+ \beta' (A'A) \beta+ 2\lambda \|\Xi_{A} \beta\|_{1} \Bigr\} ,
$$
where $(\widetilde{X'X})^{-1}$ is exactly $(X'X)^{-1}$ if $(X'X)$ is invertible, and any pseudo-inverse of this
matrix otherwise, and where $\Xi_{A}$ is a diagonal matrix whose $(j,j)$-th coefficient is $\xi_{j}^{\frac{1}{2}}(A)$
with $ \xi_{j}(A) = \frac{1}{n}[ (A'A) (\widetilde{X'X})^{-1} (A'A)]_{j,j}$.
\end{dfn}

\begin{rmk}
\label{estrmk}
Equivalently we have
$$
 \hat{\beta}_{A,\lambda} \in\arg\min_{\beta \in \mathds{R}^{p}}
\Bigl\{ \left\|\tilde{Y}_{A}-A\beta\right\|_{2}^{2} + 2\lambda \|\Xi_{A} \beta\|_{1} \Bigr\} ,
$$
where $ \tilde{Y}_{A} = A(\widetilde{X'X})^{-1}X'Y$.
\end{rmk}

Actually, we are going to consider three particular cases of this estimator in this work, depending on
the objective of the statistician:
\begin{itemize}
\item {\bf denoising objective:} the LASSO, denoted here by $\hat{\beta}_{X,\lambda}$, given by
\begin{multline*}
\hat{\beta}_{X,\lambda} \in
 \arg\min_{\beta \in \mathds{R}^{p}} \left\{\left\|Y-X\beta\right\|_{2}^{2} + 2\lambda \|\beta\|_{1} \right\}
\\ =
 \arg\min_{\beta \in \mathds{R}^{p}} \left\{-2Y'X\beta + \beta'X'X\beta + 2\lambda \|\beta\|_{1} \right\}
\end{multline*}
(note that in this case, $\Xi_{X}=I$ since $X$ is normalized);
\item {\bf transductive objective:} the Transductive LASSO, denoted here by $\hat{\beta}_{\sqrt{n/m}Z,\lambda}$,
given by
$$
\hat{\beta}_{\sqrt{\frac{n}{m}}Z,\lambda} \in\arg\min_{\beta \in \mathds{R}^{p}}
\Bigl\{  \frac{n}{m}\left\|\tilde{Y}_{Z}-Z\beta\right\|_{2}^{2}
+ 2\lambda \|\Xi_{\frac{n}{m}Z'Z} \beta\|_{1} \Bigr\};
$$
\item {\bf estimation objective:} $\hat{\beta}_{\sqrt{n}I,\lambda}$, defined by
$$
\hat{\beta}_{\sqrt{n}I,\lambda} \in\arg\min_{\beta \in \mathds{R}^{p}}
\Bigl\{ n \left\|\tilde{Y}_{I}-\beta\right\|_{2}^{2} + 2\lambda \|\Xi_{\sqrt{n}I} \beta\|_{1} \Bigr\}.
$$
\end{itemize}

Let us give the analogous definition for an extension of the Dantzig Selector.

\begin{dfn}
\label{est2}
For a given parameter $\lambda>0$ and any matrix $A$ such that $Ker(A)=Ker(X)$,
we consider the estimator given by
\begin{equation}
\label{eq:DantzigGenerlll}
\tilde{\beta}_{A,\lambda}
=\left\{
\begin{array}{l}
\arg\min_{\beta \in \mathds{R}^{p}} \left\|\beta\right\|_{1}
\\
\\
s. t. \left\|\Xi_{A}^{-1} A'A ((\widetilde{X'X})^{-1}X'Y-\beta)\right\|_{\infty} \leq \lambda.
\end{array}
\right.
\end{equation}
\end{dfn}

Here again, we are going to consider three cases, for $A=X$, $A=\sqrt{n/m}Z$ and $A=\sqrt{n}I$,
and it is easy to check that for $A=X$ we have exactly the usual definition of the Dantzig
Selector (Program \ref{DS-def}). Moreover, here again, note that we can rewrite this estimator:
\begin{equation*}
\tilde{\beta}_{A,\lambda}
=\left\{
\begin{array}{l}
\arg\min_{\beta \in \mathds{R}^{p}} \left\|\beta\right\|_{1}
\\
\\
s. t. \left\|\Xi_{A}^{-1} A'(\tilde{Y}_{A}-A \beta)\right\|_{\infty} \leq \lambda.
\end{array}
\right.
\end{equation*}

The following proposition provides an interpretation of our estimators when $A=  \sqrt{n}I$.

\begin{prop}
\label{propLSE}
Let us assume that $(X'X)$ is invertible. Then $\hat{\beta}_{\sqrt{n}I,\lambda}=\tilde{\beta}_{\sqrt{n}I,\lambda}$
and this
is a soft-thresholded
least-square estimator:
let us put $\hat{\beta}^{LSE} = (X'X)^{-1}X'Y$ then $\hat{\beta}_{\sqrt{n}I,\lambda}$
is the vector obtained by replacing the $j$-th
coordinate $b_{j}=\hat{\beta}^{LSE}_{j}$
of $\hat{\beta}^{LSE}$ by $sgn(b_{j})\left(|b_{j}|-\lambda\xi_{j}(nI)/n\right)_{+}$, where we use the
standard notation $sgn(x)=+1$ if $x\geq 0$, $sgn(x)=-1$ if $x< 0$ and $(x)_{+}=\max(x,0)$.
\end{prop}
Proposition~\ref{dualLASSO} deals with a dual definition of the estimator $\hat{\beta}_{A,\lambda}$.
\begin{prop}
\label{dualLASSO}
When $\Ker(A)=\Ker(X)$, the solutions $\beta$ of the following program:
\begin{equation*}
\left\{
\begin{array}{l}
\arg\min_{\beta \in \mathds{R}^{p}} \left\|A \beta\right\|_{2}^{2}
\\
\\
s. t. \left\|\Xi_{A}^{-1} A'((\tilde{Y}_{A}-A\beta)\right\|_{\infty} \leq \lambda
\end{array}
\right.
\end{equation*}
all satisfy $X\beta = X\hat{\beta}_{A,\lambda}$ and $A\beta = A\hat{\beta}_{A,\lambda}$.
\end{prop}
Proofs can be found in Section \ref{proofs}, page \pageref{proofs}.

\subsection{Theoretical results}

Let us first introduce our main assumption. This assumption is stated with a given $p\times p$ matrix $M$ and a given real number $x>0$.

\begin{description}
\item[Assumption $H(M,x)$:] there is a constant $c(M)>0$ such that, for any
$\alpha\in\mathds{R}^{p}$ such that
$ \sum_{j:\beta^{*}_{j} = 0}
          \xi_{j}(M)\left|\alpha_{j}\right| \leq x \sum_{j:\beta^{*}_{j} \neq 0}
                    \xi_{j}(M) \left|\alpha_{j}\right| $
we have
\begin{equation}
\label{eq:CondHypEase}
\alpha'M \alpha \geq  c(M) n \sum_{j:\beta_{j}^*\neq 0}\alpha_{j}^{2} .
\end{equation}
\end{description}

First, let us explain briefly the meaning of this hypothesis. In the case, where $M$ is invertible,
the condition
$$\alpha'M \alpha \geq c(M) n \sum_{j:\beta_{j}^*\neq 0} \alpha_{j}^{2} $$
is always satisfied for any $\alpha
\in\mathds{R}^{p}$ with $c(M)$ larger than the smallest eigenvalue of $M/n$. However, for
the LASSO, we have $M=(X'X)$ and $M$ cannot be invertible if $p>n$. Even in this case, Assumption $H(M,x)$ may still be
satisfied. Indeed, the assumption requires that Inequality~\eqref{eq:CondHypEase} holds only for a small for a small subset of $\mathds{R}^{p}$ determined by the condition $ \sum_{j:\beta^{*}_{j} = 0}
          \xi_{j}(M)\left|\alpha_{j}\right| \leq x \sum_{j:\beta^{*}_{j} \neq 0}
                    \xi_{j}(M) \left|\alpha_{j}\right| .$
For $M=(X'X)$, this assumption becomes exactly the one taken in \cite{BTWAggSOI}.
In that paper, the necessity of such an hypothesis is also discussed.

\begin{thm}
\label{mainthm}
Let us assume that Assumption $H(A'A,3)$ is satisfied and that $Ker(A)=Ker(X)$.
Let us choose $0<\eta<1$ and $\lambda = 2 \sigma \sqrt{2n\log\left(p / \eta \right)} $.
With probability at least $1-\eta$ on the draw of $Y$, we have simultaneously
$$
\left\|A\left(\hat{\beta}_{A,\lambda}-\beta^{*}\right)\right\|_{2}^{2}
\leq \frac{72 \sigma^{2}}{c(A'A)}\log\left(\frac{p}{\eta}\right) \sum_{j:\beta^{*}_{j} \neq 0} \xi_{j}(A),
$$
and
$$
\left\|\Xi_{A}\left(\hat{\beta}_{A,\lambda}-\beta^{*}\right)\right\|_{1}
\leq \frac{24 \sqrt{2} \sigma}{c(A'A)} \left(\frac{\log\left(p / \eta\right)}{n}\right)^{\frac{1}{2}}
\sum_{j:\beta^{*}_{j} \neq 0} \xi_{j}(A).
$$
\end{thm}
\noindent In particular, the first inequality gives
\begin{itemize}
\item if Assumption $H(X'X,3)$ is satisfied, with probability at least $1-\eta$,
$$
 \frac{1}{n}\left\|X\left(\hat{\beta}_{X,\lambda}-\beta^{*}\right)\right\|_{2}^{2}
      \leq \frac{72 \sigma^{2}}{n c(X'X)}
            \|\beta^*\|_{0} \log\left(\frac{p}{\eta} \right);
$$
\item if Assumption $H(\frac{n}{m}Z'Z,3)$ is satisfied, and if $Ker(Z) = Ker(X)$, with probability at least $1-\eta$,
$$
\frac{1}{m}\left\|Z\left(\hat{\beta}_{Z,\lambda}-\beta^{*}\right)\right\|_{2}^{2}
      \leq \frac{72 \sigma^{2}}{n c(\frac{n}{m} Z'Z)}
            \sum_{j:\beta_{j}^{*} \neq 0} \xi_{j}\left(\sqrt{n/m}Z\right)  \log\left(\frac{p}{\eta}\right) ;
$$
\item and if $(X'X)$ is invertible, with probability at least $1-\eta$,
$$
\left\|\hat{\beta}_{\sqrt{n}I,\lambda}-\beta^{*} \right\|_{2}^{2}
      \leq \frac{72 \sigma^{2}}{n c(nI)}
            \sum_{j:\beta_{j}^{*} \neq 0} \xi_{j}(nI)  \log \left(\frac{p}{\eta}\right)
.
$$
\end{itemize}

This result shows that each of these three estimators satisfy at least
a SI for the task it is designed for. For example, the LASSO is proved to have "good" performance
for the estimation of $X\beta^{*}$ and the Transductive LASSO is proved to have good performance for
the estimation of $Z\beta^{*}$. However we cannot assert that, for example, the LASSO performs
better than the Transductive LASSO for the estimation of $Z\beta^{*}$.

\begin{rmk}
For $A=X$, the particular case of our result applied to the LASSO is quite similar to the result
given in \cite{BTWAggSOI} on the LASSO. Actually, Theorem \ref{mainthm} can be seen as a generalization
of the result in \cite{BTWAggSOI} and it should be noted that the proof used to prove Theorem
\ref{mainthm} uses arguments introduced in \cite{BTWAggSOI}.
\end{rmk}

\begin{rmk}
As soon as $A'A$ is better determined than $X'X$, Assumption $H(A,x)$ is less restrictive
than $H(X'X,x)$. In particular,
in the case where $m>n$, Assumption $H((n/m)Z'Z,x)$ is expected to be
less restrictive than Assumption $H(X'X,x)$.
\end{rmk}

Now we give the analogous result for the 
estimator $\tilde{\beta}_{A,\lambda}$.

\begin{thm}
\label{mainthm2}
Let us assume that Assumption $H(A'A,1)$ is satisfied and that $Ker(A)=Ker(X)$.
Let us choose $0<\eta<1$ and $\lambda = 2 \sigma \sqrt{2n\log\left(p / \eta \right)} $.
With probability at least $1-\eta$ on the draw of $Y$, we have simultaneously
$$
\left\|A\left(\tilde{\beta}_{A,\lambda}-\beta^{*}\right)\right\|_{2}^{2}
\leq \frac{72 \sigma^{2}}{c(A'A)}\log\left(\frac{p}{\eta}\right) \sum_{j:\beta^{*}_{j} \neq 0} \xi_{j}(A),
$$
and
$$
\left\|\Xi_{A}\left( \tilde{\beta}_{A,\lambda}-\beta^{*}\right)\right\|_{1}
\leq \frac{ 12\sqrt{2} \sigma}{c(A'A)} \left(\frac{\log\left(p / \eta \right)}{n}\right)^{\frac{1}{2}}
\sum_{j:\beta^{*}_{j} \neq 0} \xi_{j}(A).
$$
\end{thm}

\section{An extension to the general case}
\label{hard}

In this section, we only deal with the transductive setting, $A=\sqrt{n/m}Z$.
Let us remind that in such a framework, we observe $X$ which consists of some observations $x_{i}$ associated
to labels $Y_{i}$ in $Y$, for $i\in\{1,\ldots,n\}$. Moreover we have
additional observations $x_{i}$ for $i\in\{n+1,\ldots,m\}$ with $m>n$. We also recall that $Z$ contains all the $x_{i}$ for $i\in\{1,\ldots,m\}$ and that the objective is to estimate the corresponding labels $Y_{i}$, let us put $\tilde{Y}=(Y_{1},\ldots,Y_{m})'$.

\subsection{General remarks}

Let us have look at the definition of $\hat{\beta}_{\sqrt{n/m}Z,\lambda}$, for example as given in
Remark~\ref{estrmk}:
$$
\hat{\beta}_{\sqrt{\frac{n}{m}}Z,\lambda} \in\arg\min_{\beta \in \mathds{R}^{p}}
\Bigl\{  \frac{n}{m}\left\|\tilde{Y}_{Z}-Z\beta\right\|_{2}^{2}
+ 2\lambda \|\Xi_{\frac{n}{m}Z'Z} \beta\|_{1} \Bigr\},
$$
where actually $\tilde{Y}_{Z} = Z\left(\widetilde{X'X}\right)^{-1}XY$ can be interpreted as
a preliminary estimator of $\tilde{Y}$. Hence, in any case, we propose the following procedure.\\
\noindent Let us assume that, depending on the
context, the user has a natural (and not necessary efficient) estimator of $\tilde{Y}=(Y_{1},\ldots,Y_{n+m})'$. Note this estimator $\check{Y}$.

\begin{dfn}
\label{estG1}
The Transductive LASSO is given by:
$$
\hat{\beta}_{\check{Y},\sqrt{\frac{n}{m}}Z,\lambda} \in\arg\min_{\beta \in \mathds{R}^{p}}
\Bigl\{  \frac{n}{m}\left\|\check{Y}-Z\beta\right\|_{2}^{2}
+ 2\lambda \|\Xi_{\frac{n}{m}Z'Z} \beta\|_{1} \Bigr\},
$$
and the Transductive Dantzig Selector is defined as:
\begin{equation*}
\tilde{\beta}_{\check{Y},\sqrt{\frac{n}{m}}Z,\lambda}
=\left\{
\begin{array}{l}
\arg\min_{\beta \in \mathds{R}^{p}} \left\|\beta\right\|_{1}
\\
\\
s. t. \left\|\frac{n}{m}\Xi_{\sqrt{n/m}Z}^{-1} Z'(\check{Y}-Z\beta)\right\|_{\infty} \leq \lambda.
\end{array}
\right.
\end{equation*}
\end{dfn}

In the next subsection, we propose a context where we have a natural estimator $\check{Y}$ and give a
SI on this estimator.

\subsection{An example: small labeled dataset, large unlabeled dataset}

The idea of this example is to consider the case where the examples $x_{i}$ for
$1\leq i \leq n$ are "representative" of the large populations $x_{i}$ for $1\leq i \leq m$.

Consider, $Z=(x_{1}',\ldots,x_{m}')'$ where the $x_{i}'s$ are the points of interest: we want
to estimate $\tilde{Y}=Z\beta^{*}$. However, we just have a very expensive and noisy procedure,
that, given a point $x_{i}$, returns $Y_{i}=x_{i} \beta^{*} + \varepsilon_{i}$, where the
$\varepsilon_{i}$'s are $\mathcal{N}(0,\sigma^{2})$ independent random variables. In such a case, the procedure cannot be applied for the whole dataset $Z=(x_{1}',\ldots,x_{m}')'$.
We can only make a deal with a "representative" sample of size $n$. A typical case could
be $n<p<m$.

First, let us introduce a slight modification of our main hypothesis. It is also stated with a given $p\times p$ matrix $M$ and a given real number $x>0$.

\begin{description}
\item[Assumption $H'(M,x)$:] there is a $c(M)>0$ such that, for any
$\alpha\in\mathds{R}^{p}$ such that
$ \sum_{j:\beta^{*}_{j} = 0}
          \left|\alpha_{j}\right| \leq x \sum_{j:\beta^{*}_{j} \neq 0}
                     \left|\alpha_{j}\right| $
we have
$$\alpha'M \alpha \geq c(M) n \sum_{j:\beta^*_j \neq 0} \alpha_{j}^2.$$
\end{description}

We can now state our main result.

\begin{thm}
\label{mainthm3}
Let us assume that Assumption $H'((n/m)Z'Z,1)$ is satisfied.
Let us choose $0<\eta<1$ and $\lambda_{1} = \lambda_{2} = 10^{-1} \sigma \sqrt{2n\log\left(p / \eta \right)} $.
Moreover, let us assume that
\begin{equation}
\label{hypoinf}
\forall u \in\mathds{R}^{p} \text{ with } \|u\|_{1} \leq  \|\beta^{*}\|_{1},\quad
\left\|\left((X'X)-\frac{n}{m}(Z'Z)\right) u \right\|_{\infty} < \frac{\sigma}{10}  \sqrt{2n\log\left(\frac{p}{\eta}\right)}  .
\end{equation}
Let $ \check{Y}_{\lambda_{1}} = Z \tilde{\beta}_{X,\lambda_{1}} $
be a preliminary estimator of $\tilde{Y}$, based on ths Dantzig Selector given by~\eqref{eq:DantzigGenerlll} (with $A=X$). Then define the Transductive LASSO by
\begin{equation*}
\hat{\beta}^{*}_{\frac{n}{m}Z, 20 \lambda_{2}}
=\left\{
\begin{array}{l}
\arg\min_{\beta \in \mathds{R}^{p}} \frac{n}{m} \left\|Z\beta\right\|_{2}^{2}
\\
\\
s. t. \left\|\frac{n}{m} Z'(\check{Y}_{\lambda_{1}} - Z\beta) \right\|_{\infty} \leq 20 \lambda_{2},
\end{array}
\right.
\end{equation*}
and the Transductive Dantzig Selector
\begin{equation*}
\tilde{\beta}^{*}_{\frac{n}{m}Z,\lambda_{2}}
=\left\{
\begin{array}{l}
\arg\min_{\beta \in \mathds{R}^{p}} \left\|\beta\right\|_{1}
\\
\\
s. t. \left\|\frac{n}{m} Z'(\check{Y}_{\lambda_{1}}-Z\beta) \right\|_{\infty} \leq \lambda_{2} .
\end{array}
\right.
\end{equation*}
With probability at least $1-\eta$ on the draw of $Y$, we have simultaneously
$$
\frac{1}{m}\left\|Z(\tilde{\beta}^{*}_{\frac{n}{m}Z,\lambda_{2}}-\beta^*)\right\|_{2}^{2}
\leq \frac{16 \sigma^{2}}{n c((n/m)Z'Z)}\log\left(\frac{p}{\eta}\right) \|\beta^*\|_{0},
$$
$$
\left\|\tilde{\beta}^{*}_{\frac{n}{m}Z,\lambda_{2}}-\beta^*\right\|_{1}
\leq \frac{ 8 \sigma}{c((n/m)Z'Z)} \left(\frac{\log\left( p / \eta\right)}{n}\right)^{\frac{1}{2}}
\|\beta^*\|_{0},
$$
and moreover, if $H'((n/m)Z'Z,5)$ is also satisfied,
$$
\frac{1}{m}\left\|Z(\hat{\beta}^{*}_{\frac{n}{m}Z, 20 \lambda_{2}}-\beta^*)\right\|_{2}^{2}
\leq \frac{88 \sigma^{2}}{n c((n/m)Z'Z)}\log\left(\frac{p}{\eta}\right) \|\beta^*\|_{0},
$$
$$
\left\|\hat{\beta}^{*}_{\frac{n}{m}Z, 20 \lambda_{2}}-\beta^*\right\|_{1}
\leq \frac{ 54 \sigma}{c((n/m)Z'Z)} \left(\frac{\log\left(p / \eta \right)}{n}\right)^{\frac{1}{2}}
\|\beta^*\|_{0}.
$$
\end{thm}

First, let us remark that the preliminary estimator $\check{Y}_{\lambda_{1}}$ is defined using the Dantzig Selector
$\tilde{\beta}_{X,\lambda_{1}}$.
We could give exactly the same kind of results using a the LASSO
$\hat{\beta}_{X,\lambda_{1}}$ as a preliminary estimator.

Now, let us give a look at the new hypothesis, Inequality \eqref{hypoinf}. We can interpret this
condition as the fact that the
$x_{i}$'s for $1\leq i \leq n$ are effectively representative of the wide population:
so $X'X/n$ is "not too far" from $Z'Z/m$.
We will end this section by a result that proves that this is effectively the case
in a typical situation.

\begin{prop}
\label{propmatrix}
Assume that $m=kn$ for an integer value $k\in\mathds{N}\setminus\{0,1\}$.
Let us assume that $X$ and $Z$ are build in the following way: we have a population
$\chi_{1}=(\chi_{1,1},\ldots,\chi_{1,p})\in\mathds{R}^{p}$,\ldots, $\chi_{m}\in\mathds{R}^{p}$ (the points of interest).
Then, we draw uniformly without replacement, $n$ of the $\chi_{i}$'s to be put in $X$: more
formally, but equivalently, we draw uniformly a permutation $\sigma$ of $\{1,\ldots,m\}$ and
we put $X=(x_{1}',\ldots,x_{n}')'=(\chi_{\sigma(1)}',\ldots,\chi_{\sigma(n)}')'$
and $Z=(x_{1}',\ldots,x_{m}')'=(\chi_{\sigma(1)}',\ldots,\chi_{\sigma(m)}')'$.\\
\noindent Let us assume that for any $(i,j)\in\{1,\ldots,m\}\times\{1,\ldots,p\}$, $\chi_{i,j}^2<\kappa$ for some $\kappa>0$,
and that $p\geq 2$.
Then, with probability at least $1-\eta$, for any $u\in\mathds{R}^{p}$,
$$ \left\|\left(X'X -\frac{n}{m} Z'Z\right)u \right\|_{\infty}
  \leq \left\|u \right\|_{1} \frac{2 \kappa k}{k-1} \sqrt{2 \log\frac{p}{\eta}} .$$
In particular, if we have
$$ \|u\|_{1} \leq \|\beta^*\|_{1} \text{ and } \kappa \leq \frac{k-1}{10 \, k} \frac{\sigma}{\|\beta^*\|_{1}} $$
then we have
$$ \left\|\left(X'X -\frac{n}{m} Z'Z\right)u\right\|_{\infty} \leq
\sigma \sqrt{2n\log\left(\frac{p}{\eta}\right)}
. $$
\end{prop}
Let us just mention that the assumption $m=kn$ is not restrictive. It has been introduced for the sake of simplicity.

\section{Experimental results}

\label{simu}

\noindent {\bf Implementation.} Since the paper of Tibshirani \cite{Tibshirani-LASSO},
several effective algorithms to compute the LASSO have been proposed
and studied (for instance Interior Points methods \cite{interior}, LARS \cite{Efron-LARS}, Pathwise
Coordinate Optimization \cite{PCO}, Relaxed Greedy Algorithms \cite{Barron2}).
For the Dantzig Selector, a linear method was proposed in the first paper \cite{Dantzig}.
The LARS algorithm was also successfully extended in \cite{DASSO} to compute
the Dantzig Selector.\\
\noindent Then there are many algorithms to compute $\hat{\beta}_{A,\lambda}$ and $\tilde{\beta}_{A,\lambda}$, when $A=X$. Thanks to Proposition
\ref{propLSE}, it is also clear that we can easily find an efficient algorithm for the case $A=\sqrt{n}I$.\\
\noindent The general form of the estimators $\hat{\beta}_{A,\lambda}$ and $\tilde{\beta}_{A,\lambda}$ given by
Definitions~\ref{est} and~\ref{est2}, allows to use one of the algorithms mentioned previously
to compute our estimator in two cases.
For example, from Remark~\ref{est}, we have:
$$
\hat{\beta}_{A,\lambda} \in\arg\min_{\beta \in \mathds{R}^{p}}
\Bigl\{ \left\|\tilde{Y}_{A}-A\beta\right\|_{2}^{2} + 2\lambda \|\Xi_{A} \beta\|_{1} \Bigr\} ,
$$
then we just have to compute $\tilde{Y}_{A}$, to put $B = A\Xi_{A}^{-1}$, to use any program
that computes the LASSO to determine
$$ \hat{\gamma} \in \arg\min_{\gamma \in \mathds{R}^{p}}
\Bigl\{ \left\|\tilde{Y}_{A}-B\gamma\right\|_{2}^{2} + 2\lambda \|\gamma\|_{1} \Bigr\} $$
and then to put $\hat{\beta}_{A,\lambda} = \Xi_{A}^{-1}\gamma$.

In the rest of this section, we compare the LASSO and the transductive LASSO on the classical toy
example introduced by Tibshirani \cite{Tibshirani-LASSO} and used as a benchmark.\\

\noindent{\bf Data description.} In the model proposed by Tibshirani, we have
$$ Y_{i} = x_i \beta^* + \varepsilon_i $$
for $i\in\{1,\ldots,n\}$, $\beta^* \in\mathds{R}^{p}$ and the $\varepsilon_{i}$ are i.i.d. $\mathcal{N}(0,\sigma^2)$.
Finally, the $(x_{i})_{i\in\{1,\ldots,m\}}$ are generated from a probability distribution: they are independent and
identically distributed
$$
x_{i}
  \sim \mathcal{N} \left(
\left(\begin{array}{c} 0 \\ \vdots \\ 0 \end{array}\right)
,
\left(\begin{array}{c c c c c} 1 & \rho & \dots & \dots & \rho^{p-1}
              \\            \rho & 1 & \rho & \dots & \rho^{p-2}
              \\            \vdots & \ddots & \ddots & \ddots & \vdots
              \\           \rho^{p-2} & \dots & \rho & 1 & \rho
               \\           \rho^{p-1} & \dots & \dots & \rho & 1 \end{array}\right)
\right),
$$
for a given $\rho\in]-1,1[$.

As in \cite{Tibshirani-LASSO}, we set $p=8$. In a first experiment, we take
$(n,m)=(7,10)$, $\rho=0.5$, $\sigma=1$ and $\beta^*=(3,1.5,0,0,2,0,0,0)$ ("sparse").
Then, in order to check the robustness of the results, we consider
successively $\rho=0.5$ by $\rho=0.9$ (correlated variables),
$\sigma=1$ by $\sigma=3$ (noisy case), $\beta^*=(3,1.5,0,0,2,0,0,0)$ by $\beta^*=(5,0,0,0,0,0,0,0)$ ("very
sparse" case), $(n,m)=(7,10)$ by $(n,m)=(7,20)$ (larger unlabeled set), $(n,m)=(20,30)$ ($p<n$, easy case) and finally $(n,m)=(20,120)$.

We use the version of the Transductive LASSO proposed in Section \ref{hard}:
for a given $\lambda_{1}$, we first compute the LASSO estimator $\hat{\beta}_{X,\lambda_{1}}$.
In the sequel, the Transductive LASSO is given by
\begin{equation*}
\hat{\beta}^{TL}(\lambda_{1},\lambda_{2})= \left\{
\begin{array}{l}
\arg\min_{\beta \in \mathds{R}^{p}} \frac{n}{m} \left\|Z\beta\right\|_{2}^{2}
\\
\\
s. t. \left\|\frac{n}{m} Z'(Z\hat{\beta}_{X,\lambda_{1}} - Z\beta) \right\|_{\infty} \leq \lambda_{2},
\end{array}
\right.
\end{equation*}
for a given $\lambda_{2}$. We compare this two step procedure with the procedure obtained
using the usual LASSO only: $\hat{\beta}^{L}(\lambda)=\hat{\beta}_{X,\lambda}$
for a given $\lambda$ that may differ from $\lambda_{1}$.
In both cases, the solutions are computed using PCO algorithm. We compute $\hat{\beta}^{L}(\lambda)$
and $\hat{\beta}^{TL}(\lambda_{1},\lambda_{2})$ for $(\lambda,\lambda_{1},\lambda_{2})\in\Lambda^3$
where $\Lambda^3 = \{1.2^{k},k=-50,-49,\ldots,30\}$. In the next subsection, we examine the performance
of each estimator according to the value of the regularization parameters.\\

\noindent {\bf Results.} We illustrate here some of the results obtained in the considered cases.\\

\noindent{\it Case $(n,m)=(7,10)$, $\rho=0.5$, $\sigma=1$ and $\beta^*$ "sparse":}\\
\noindent We simulated $100$ experiments and studied the distribution of
$$ PERF(X) = \frac{\min_{(\lambda_{1},\lambda_{2})\in\Lambda^{2}} \| X(\hat{\beta}^{TL}(\lambda_{1},\lambda_{2})-\beta^*)
      \|_{2}^{2} }{ \min_{\lambda\in\Lambda} \| X(\hat{\beta}^{L}(\lambda)-\beta^*)
               \|_{2}^{2}}, $$
$$ PERF(Z) = \frac{\min_{(\lambda_{1},\lambda_{2})\in\Lambda^{2}} \| Z(\hat{\beta}^{TL}(\lambda_{1},\lambda_{2})-\beta^*)
      \|_{2}^{2} }{ \min_{\lambda\in\Lambda} \| Z(\hat{\beta}^{L}(\lambda)-\beta^*)
               \|_{2}^{2}}, $$
and
$$ PERF(I) = \frac{\min_{(\lambda_{1},\lambda_{2})\in\Lambda^{2}} \| \hat{\beta}^{TL}(\lambda_{1},\lambda_{2})-\beta^*
      \|_{2}^{2} }{ \min_{\lambda\in\Lambda} \| \hat{\beta}^{L}(\lambda)-\beta^*
               \|_{2}^{2}}, $$
over all the experiments.

For example, we plot (Figure~\ref{fig:hist}) the histogram of
$PERF(X)$ (actually, the three distributions where quite similar).
\begin{figure}[ht]
\begin{center}
\includegraphics[width=7cm, height=7cm]{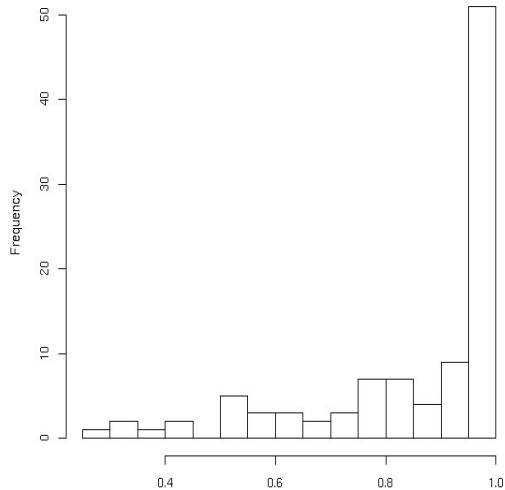}
\caption{\label{fig:hist} Histogram of $PERF(X)$ with $(n,m)=(7,10)$, $\rho=0.5$, $\sigma=1$ and $\beta^*=(3,1.5,0,0,2,0,0,0)$.}
\end{center}
\end{figure}
We observe that in $50\%$ of the simulations, $\min_{(\lambda_{1},\lambda_{2})\in\Lambda^{2}}
\| X(\hat{\beta}^{TL}(\lambda_{1},\lambda_{2})-\beta^*) \|_{2}^{2} = \min_{(\lambda_{1},0)\in\Lambda^{2}}
\| X(\hat{\beta}^{TL}(\lambda_{1},0)-\beta^*) \|_{2}^{2} =
\min_{\lambda\in\Lambda} \| X(\hat{\beta}^{L}(\lambda)-\beta^*)\|_{2}^{2}$. In these cases, the Transductive LASSO
does not improve at all the LASSO. But in the others $50\%$, the Transductive LASSO actually improve the LASSO,
and the improvement is sometimes really important. We give an overview of the results in Table~\ref{tab:frqVSsigma}.\\

\begin{table}[t]
\caption{Evaluation of the mean $ME$ and the quantile $Q_3$ of order $0.3$ of $PERF(I)$, $PERF(X)$ and $PERF(Z)$. In these experiments, $\sigma$ always equals $1$. The case {\it sparse} corresponds to $\beta^*=(3,1.5,0,0,2,0,0,0)$ while the case {\it very sparse} corresponds to $\beta^*=(5,0,0,0,0,0,0,0)$.}
\label{tab:frqVSsigma} 
\begin{center}
\begin{sc}
\begin{tabular}{|l|c|c|c||c|c||c|c||c|r|}
\cline{5-10}
 \multicolumn{4}{c||}{} & \multicolumn{2}{c||}{$PERF(I)$}    & \multicolumn{2}{c||}{$PERF(X)$}    &\multicolumn{2}{c|}{$PERF(Z)$}    \\
\hline
$\beta^*$      &  $(n,m)$ & $\rho$ & $\sigma$ &  $ME$   &  $Q_3$  & $ME$   &  $Q_3$  & $ME$   &  $Q_3$  \\
\hline
\hline
very sparse  & $(7,10)$ &  $0.5$  & $1$ &  $0.74$  & $0.71$  & $0.76$  & $0.71$ & $0.75$  &  $0.70$  \\
\hline
sparse         & $(7,10)$ &  $0.5$  & $1$ &  $0.83$  & $0.76$  & $0.86$  & $0.80$    & $0.88$  &  $0.88$ \\
\hline
sparse         & $(7,20)$ &  $0.5$  & $1$ & $0.84$   & $0.79$  &  $0.84$ & $0.81$    & $0.88$  &  $0.89$    \\
\hline
sparse         & $(20,30)$&  $0.5$   & $1$  &  $0.91$  &  $0.90$ & $0.93$  &   $0.93$ & $0.93$  &  $0.95$    \\
\hline
sparse         & $(20,30)$ & $0.9$   & $1$ &  $0.91$  &  $0.93$ &  $0.94$ &  $0.95$ & $0.93$  &  $0.96$       \\
\hline
sparse         & $(20,30)$ & $0.5$   & $3$ &  $0.90$  &  $0.89$ &  $0.92$ &  $0.92$ & $0.92$  &  $0.93$       \\
\hline
\end{tabular}
\end{sc}
\end{center}
\vskip -0.1in
\end{table}

\noindent{\it The other cases :}\\
\noindent The following conclusions emerge of the experiments:
first, $\beta^*=(5,0,\dots,0)$ leads to a more significative improvement of
the Transductive LASSO compared to the LASSO (Table \ref{tab:frqVSsigma}).
This good performance of the Transductive LASSO can also be observed when $(n,m)=(7,10)$ and $(n,m)=(7,20)$.
However in the case $n>p$ (easy case), i.e., $(n,m)=(20,30)$ and $(n,m)=(20,120)$, the improvement
of the Transductive LASSO with respect to the LASSO becomes less significant (Table~\ref{tab:frqVSsigma}).\\
\noindent Finally, $\rho$ and $\sigma$ have of course a significant influence
on the performance of the LASSO. However these parameters do not seem to have any influence on
the relative performance of the Transductive LASSO with respect to the LASSO
(see for instant the three last rows in Table~\ref{tab:frqVSsigma}, where we kept $(n,m)=(20,30)$).\\
\noindent Quite surprisingly, the relative performance of both estimators does not strongly depend on the
estimation objective $\beta^*$, $X\beta^*$ or $Z\beta^*$, but on the particular experiment we deal with.
According to the realized study and for all the objectives, the
Transductive LASSO performs better than the LASSO in about $50\%$ of the experiments. Otherwise, $\lambda_{1}=0$ is the optimal tuning parameter and then, the LASSO and the
Transductive LASSO are equivalent.

Also surprising is that as often as not, the minimum in
$$
\min_{(\lambda_{1},\lambda_{2})\in\Lambda^{2}}
\| X(\hat{\beta}^{TL}(\lambda_{1},\lambda_{2})-\beta^*) \|_{2}^{2} < \min_{(\lambda_{1},0)\in\Lambda^{2}}
\| X(\hat{\beta}^{TL}(\lambda_{1},0)-\beta^*) \|_{2}^{2} ,$$
{\it does not significantly depend on} $\lambda_{1}$ for a very large range of
values $\lambda_{1}$. This is quite interesting for a practitioner as it means that when we use
the Transductive LASSO, we deal with only a singular unknown tuning parameter (that is $\lambda_{2}$) and not two.\\

\noindent{\bf Discussion on the regularization parameter.}
Finally, we would like to point out the importance of the tuning parameter $\lambda$ (in a general term).
Figure~\ref{fig:perf} illustrates a graph of a typical experiment.
There are two curves on this graph, that represent the quantities $(1/n)\| X(\hat{\beta}^{L}(\lambda)-\beta^*)\|_{2}^{2}$ and $(1/m)\| Z(\hat{\beta}^{L}(\lambda)-\beta^*)\|_{2}^{2}$ with respect to $\lambda$. 
We observe that both functions do not
reach their minimum value for the same value of $\lambda$ (the minimum is highlighted on the graph by a dot), even if these
minimum are quite close.

\begin{figure}[ht]
\begin{center}
\includegraphics[width=7cm, height=7cm]{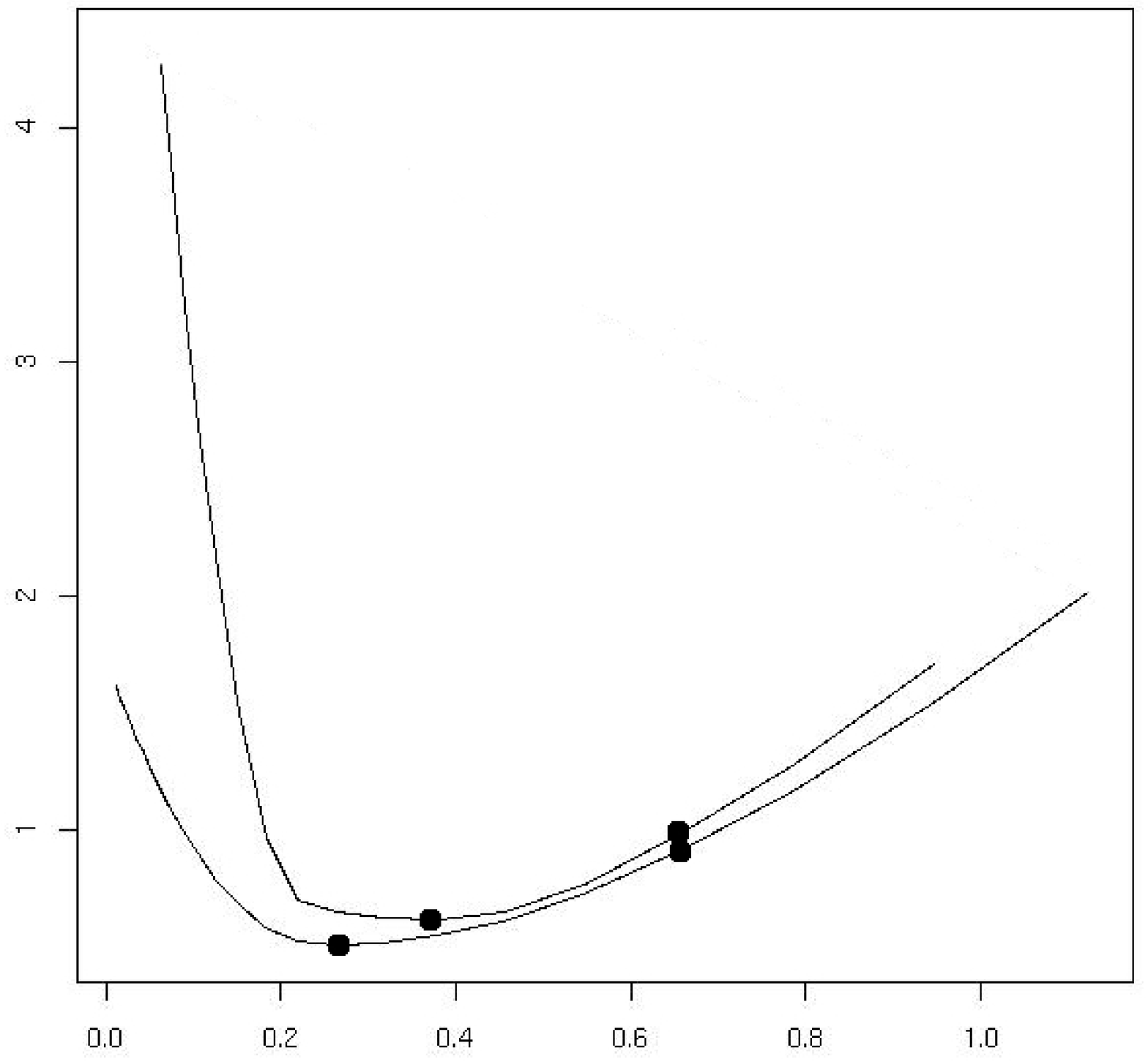}
\caption{\label{fig:perf} Performance vs. $\lambda$.}
\end{center}
\end{figure}

Since we consider variable selection methods, the identification of the true support
$\{j:\,\beta_j^*\neq 0\}$ of the vector $\beta^*$ is also in concern. One expects that the
estimator $\hat{\beta}$ and the true vector $\beta^*$ share the same support at least when
$n$ is large enough. This is known as the variable selection consistency problem and it
has been considered for the LASSO estimator in several works (see
\cite{Bunea_consist,MeinshBulhmConsistLasso,MeinYuSelect,WainSelection,BiYuConsistLasso}).
Recently, \cite{KarimNormSup} provided the variable selection consistency of the Dantzig
Selector. Other popular selection procedures, based on the LASSO estimator, such as the
Adaptive LASSO~\cite{AdapLassoZou}, the SCAD~\cite{FanLiScad}, the S-LASSO~\cite{Mo7SLasso}
and the Group-LASSO~\cite{BachGpLasso}, have also been studied under a variable selection point of view.
Following our previous work \cite{L1MOH}, it is possible to provide such results for
the Transductive LASSO.

The variable selection task has also been illustrated in Figure~\ref{fig:perf}. We reported the minimal value of $\lambda$ for which
the LASSO estimator identifies correctly the non zero components of $\beta^*$. This value of
$\lambda$ is quite different from the values that minimizes the prediction losses.
This observation is recurrent
in almost all the experiments: the estimation $X\beta^*$, $Z\beta^*$ and the support of $\beta^*$ are three
different objectives and have to be treated separately. We cannot expect in general to find a choice for $\lambda$ which makes the LASSO, for instance, has good performance for all the mentioned objective simultaneously.

\section{Conclusion}

In this paper, we propose an extension of the LASSO and the Dantzig Selector for which we
provide theoretical results with less restrictive hypothesis than in previous works.
These estimators have a nice interpretation in terms of transductive prediction. Moreover, we study the practical performance of the proposed transductive estimators on simulated data. It turns out that the benefit using such methods is emphasized when the model is sparse and particularly when the samples sizes ($n$ labeled points and $m$ unlabeled points) and dimension $p$ are such that $n<p<m$.

\section{Proofs}
\label{proofs}

In this section, we state the proofs of our main results.

\subsection{Proof of Propositions \ref{propLSE} and \ref{dualLASSO}}

\begin{proof}[Proof of Proposition \ref{propLSE}]
Let us assume that $(X'X)$ is invertible. Then just remark that the criterion minimized by $\hat{\beta}_{\sqrt{n}I,\lambda}$
is just
$$
      n\left\|\hat{\beta}^{LSE}-\beta\right\|^{2}_{2}+ 2\lambda \|\Xi_{nI} \beta\|_{1}
      =  \sum_{j=1}^{p} \left\{\left[\hat{\beta}^{LSE}_{j}-\beta_{j}\right]^{2}
                     + \frac{2\lambda\xi_{j}(\sqrt{n}I)}{n}|\beta_{j}| \right\}.
$$
So we can optimize with respect to each coordinate $\beta_{j}$ individually. It is quite easy to check
that the solution is, for $\beta_{j}$,
$$ sgn\left(\hat{\beta}^{LSE}_{j}\right)\left(\left|\hat{\beta}^{LSE}_{j}\right|
     -\frac{\lambda\xi_{j}(\sqrt{n}I)}{n}\right)_{+}. $$
The proof for $\hat{\beta}_{\sqrt{n}I,\lambda}$ is also easy as it solves
\begin{equation*}
\left\{
\begin{array}{l}
\arg\min_{\beta \in \mathds{R}^{p}} \left\|\beta\right\|_{1}
\\
\\
s. t. \left\|n\Xi_{nI}^{-1} (\hat{\beta}^{LSE}-\beta)\right\|_{\infty} \leq \lambda.
\end{array}
\right.
\end{equation*}
\end{proof}

\begin{proof}[Proof of Proposition \ref{dualLASSO}]
Let us write the Lagrangian of the program
\begin{equation*}
\left\{
\begin{array}{l}
\arg\min_{\beta \in \mathds{R}^{p}} \left\|A \beta\right\|_{2}^{2}
\\
\\
s. t. \left\|\Xi_{A}^{-1} (A'A) ((\widetilde{X'X})^{-1}X'Y-\beta)\right\|_{\infty} \leq \lambda,
\end{array}
\right.
\end{equation*}
\begin{multline*}
\mathcal{L}(\beta,\gamma,\mu) = \beta (Z'Z) \beta
     + \gamma' \left[ \Xi_{A}^{-1} (A'A) ((\widetilde{X'X})^{-1}X'Y-\beta)  - \lambda E \right]
\\
     + \mu' \left[ \Xi_{A}^{-1} (A'A) (\beta -(\widetilde{X'X})^{-1}X'Y) - \lambda E \right]
\end{multline*}
with $E=(1,\ldots,1)'$, and for any $j$, $\gamma_{j}\geq 0$, $\mu_{j} \geq 0$ and $\gamma_{j}\mu_{j}=0$.
Any solution $\underline{\beta}=\underline{\beta}(\gamma,\mu)$
must satisfy
$$
0 = \frac{\partial \mathcal{L}}{\partial \beta}(\underline{\beta},\lambda,\mu)
= 2 \underline{\beta} (A'A) + (\gamma-\mu)  \Xi_{A}^{-1} (A'A)
$$
so
\begin{equation*}
 (A'A)\underline{\beta} = (A'A)\Xi^{-1}_{A} \frac{\mu-\gamma}{2}.
\end{equation*}
Note that the conditions $\gamma_{j}\geq 0$, $\mu_{j} \geq 0$ and $\gamma_{j}\mu_{j}=0$
means that there is a $\zeta_{j}\in\mathds{R}$ such that $\zeta_{j}=\xi_{j}^{\frac{1}{2}}(A)
(\mu_{j}-\gamma_{j})/2$,
$|\zeta_{j}| = \xi_{j}^{\frac{1}{2}}(A)(\gamma_{j}+\mu_{j})/2$, and so
$\gamma_{j}=2(\zeta_{j}/\xi_{j}^{\frac{1}{2}}(A))_{-}$
and $\mu_{j}=2(\zeta_{j}/\xi_{j}^{\frac{1}{2}}(A))_{+}$,
where $(a)_{+}=max(a;0)$ and  $(a)_{-}=max(-a;0)$. Let also $\zeta$ denote the vector which $j$-th component
is exactly $\zeta_{j}$, we obtain
$$
 (A'A)\underline{\beta} = (A'A)\zeta,
$$
or, using the condition $\Ker(A)=\Ker(X)$, $X\underline{\beta} = X\zeta$ and $A\underline{\beta} = A\zeta$.
This leads to
$$
\mathcal{L}(\underline{\beta},\gamma,\mu)
=
-2Y'X(\widetilde{X'X})^{-1} (A'A) \zeta
+ \zeta' (A'A) \zeta+ 2\lambda \|\Xi_{A} \zeta\|_{1},
$$
and note that the first order condition also implies that $\gamma$ and $\mu$
(and so $\zeta$) maximize $\mathcal{L}$. This ends the proof.
\end{proof}

\subsection{A useful Lemma}

The following lemma will be used in the proofs of Theorems \ref{mainthm} and \ref{mainthm2}.

\begin{lemma}
\label{usefullemma}
Let us put $\varepsilon = (\varepsilon_{1},\ldots,\varepsilon_{n})'$. 
If $\Ker(A)=\Ker(X)$ we have, with probability at least $1-\eta$,
$$ \forall j\in\{1,\ldots,p\},\left| \left[A'A (\widetilde{X'X})^{-1} X'\varepsilon\right]_{j} \right|
                 \leq \xi_{j}(A) \sigma \sqrt{2n\log\frac{p}{\eta}} ,$$
or, in other words,
$$
\|\Xi^{-1}_{A} (A'A)((\widetilde{X'X})^{-1}X'Y-\beta^{*})\|_{\infty} \leq \sigma \sqrt{2n\log\frac{p}{\eta}}.
$$
\end{lemma}

\begin{proof}[Proof of the lemma]
By definition,
$\varepsilon\sim \mathcal{N}(0,\sigma^{2} I)$ and
so
$$(A'A) (\widetilde{X'X})^{-1} X' \varepsilon \sim \mathcal{N}(0,
\sigma^{2} (A'A) (\widetilde{X'X})^{-1} (A'A)).$$
So, for all $j$, $[(A'A) (\widetilde{X'X})^{-1} X' \varepsilon]_{j}$ comes
from a $ \mathcal{N}(0,\sigma^{2} \xi_{j}^{2}(A))$ distribution.
This implies the first point, the second one is trivial using $Y=X\beta^{*}+\varepsilon$.
\end{proof}

\subsection{Proof of Theorems \ref{mainthm} and \ref{mainthm2}}

\begin{proof}[Proof of Theorem \ref{mainthm}]
By definition of $ \hat{\beta}_{A,\lambda}$ we have
\begin{multline*}
-2Y'X(\widetilde{X'X})^{-1} (A'A) \hat{\beta}_{A,\lambda}
   + \left(\hat{\beta}_{A,\lambda}\right)' (A'A)  \hat{\beta}_{A,\lambda}
       + 2\lambda \|\Xi_{A'A} \hat{\beta}_{A,\lambda}\|_{1}
  \\
 \leq   2Y'X(\widetilde{X'X})^{-1} (A'A)\beta^{*} + (\beta^{*})' (A'A) \beta^{*} + 2\lambda \|\Xi_{A} \beta^{*}\|_{1}.
\end{multline*}
Since $Y=X\beta^{*}+\varepsilon$, we obtain
\begin{multline*}
2 (\beta^{*})'X'X(\widetilde{X'X})^{-1} (A'A) \left( \beta^{*} - \hat{\beta}_{A,\lambda}\right)
   + \left(\hat{\beta}_{A,\lambda}\right)' (A'A)  \hat{\beta}_{A,\lambda}
     - (\beta^{*})' (A'A) \beta^{*}
\\
+ 2 \varepsilon' X(\widetilde{X'X})^{-1} (A'A) \left( \beta^{*} - \hat{\beta}_{A,\lambda}\right)
 \leq 2\lambda \|\Xi_{A} \beta^{*}\|_{1} - 2\lambda \|\Xi_{A} \hat{\beta}_{A,\lambda}\|_{1}.
\end{multline*}
Now, if $\Ker(X)=\Ker(A)$ then we have $X'X(\widetilde{X'X})^{-1} (A'A) = (A'A)$ and then the previous inequality leads to
\begin{multline}
\label{eqproof1}
\left(\beta^{*} - \hat{\beta}_{A,\lambda}\right)' (A'A) \left( \beta^{*} - \hat{\beta}_{A,\lambda}\right)
 \\
\leq 2 \varepsilon' X(\widetilde{X'X})^{-1} (A'A) \left(\hat{\beta}_{A,\lambda} - \beta^{*} \right)
 + 2\lambda \|\Xi_{A} \beta^{*}\|_{1} - 2\lambda \|\Xi_{A} \hat{\beta}_{A,\lambda}\|_{1}.
\end{multline}
Now we have to work on the term $2 \varepsilon' X(\widetilde{X'X})^{-1}(A'A)\left(\hat{\beta}_{A,\lambda}-\beta^{*}\right)$.
Note that
\begin{eqnarray*}
2 \varepsilon' X(\widetilde{X'X})^{-1} (A'A) \left(\hat{\beta}_{A,\lambda}-\beta^{*}\right)
&
 =
&
2 \sum_{j=1}^{p} \left(\hat{\beta}_{A,\lambda}-\beta^{*}\right)_{j}
           \left[(A'A) (\widetilde{X'X})^{-1} X'\varepsilon\right]_{j}
\\
& 
\leq
&
2 \sum_{j=1}^{p} \left|\left(\hat{\beta}_{A,\lambda}-\beta^{*}\right)_{j}\right|
                    \left| \left[(A'A) (\widetilde{X'X})^{-1} X'\varepsilon\right]_{j} \right|
\\
&
\leq
&
2 \sigma\sqrt{2n\log\left(\frac{p}{\eta}\right)}
    \sum_{j=1}^{p} \xi_{j}^{\frac{1}{2}}(A) \left|\left(\hat{\beta}_{A,\lambda}\right)_{j}-\beta^{*}_{j}\right|
\end{eqnarray*}
with probability at least $1-\eta$, by Lemma \ref{usefullemma}.
We plug this result into Inequality~\eqref{eqproof1} (and replace $\lambda$ by its
value $2\sigma\sqrt{2n\log(p/\eta)}$) to obtain
\begin{multline*}
\left(\beta^{*} - \hat{\beta}_{A,\lambda}\right)' (A'A) \left( \beta^{*} - \hat{\beta}_{A,\lambda}\right)
 \\
\leq
2\sigma\sqrt{2n\log\left(\frac{p}{\eta}\right)}  \sum_{j=1}^{p} \xi_{j}^{\frac{1}{2}}(A) \Biggl\{
 \left|\left(\hat{\beta}_{A,\lambda}\right)_{j}-\beta^{*}_{j}\right|
  + 2 \left(\left|\beta^{*}_{j}\right| - \left|\left(\hat{\beta}_{A,\lambda}\right)_{j}\right|\right)
\Biggr\}
\end{multline*}
and then
\begin{eqnarray}
\label{STEP}
& & \left(\beta^{*} - \hat{\beta}_{A,\lambda}\right)' (A'A) \left( \beta^{*} - \hat{\beta}_{A,\lambda}\right) 
\nonumber \\
&&
+ 
2 \sigma\sqrt{2n\log\left(\frac{p}{\eta}\right)}  \sum_{j=1}^{p} \xi_{j}^{\frac{1}{2}}(A)
 \left|\left(\hat{\beta}_{A,\lambda}\right)_{j}-\beta^{*}_{j}\right| \nonumber
\\
& 
\leq 
&
4\sigma\sqrt{2n\log\left(\frac{p}{\eta}\right)}  \sum_{j=1}^{p} \xi_{j}^{\frac{1}{2}}(A) \Biggl\{
 \left|\left(\hat{\beta}_{A,\lambda}\right)_{j}-\beta^{*}_{j}\right|
 + 
\left|\beta^{*}_{j}\right| - \left|\left(\hat{\beta}_{A,\lambda}\right)_{j}\right|
\Biggr\}
\nonumber
\\
& 
= 
&
4\sigma\sqrt{2n\log\left(\frac{p}{\eta}\right)}  \sum_{j:\beta^{*}_{j} \neq 0} \xi_{j}^{\frac{1}{2}}(A) \Biggl\{
 \left|\left(\hat{\beta}_{A,\lambda}\right)_{j}-\beta^{*}_{j}\right|
  +
\left|\beta^{*}_{j}\right| - \left|\left(\hat{\beta}_{A,\lambda}\right)_{j}\right|
\Biggr\} 
\nonumber
\\
&
\leq
&
8\sigma\sqrt{2n\log\left(\frac{p}{\eta}\right)}  \sum_{j:\beta^{*}_{j} \neq 0} \xi_{j}^{\frac{1}{2}}(A)
 \left|\left(\hat{\beta}_{A,\lambda}\right)_{j}-\beta^{*}_{j}\right|.
\end{eqnarray}
This implies, in particular, that $\beta^{*} - \hat{\beta}_{A,\lambda}$ is an admissible vector $\alpha$
in Assumption $H(A'A,3)$ because
$$
 \sum_{j=1}^{p} \xi_{j}^{\frac{1}{2}}(A)
 \left|\left(\hat{\beta}_{A,\lambda}\right)_{j}-\beta^{*}_{j}\right|
 \leq 4\sum_{j:\beta^{*}_{j} \neq 0} \xi_{j}^{\frac{1}{2}}(A)
 \left|\left(\hat{\beta}_{A,\lambda}\right)_{j}-\beta^{*}_{j}\right|.
$$
On the other hand, thanks to Inequality \eqref{STEP}, we have
\begin{multline}
\label{eq:profoBulf}
\left(\beta^{*} - \hat{\beta}_{A,\lambda}\right)' (A'A) \left( \beta^{*} - \hat{\beta}_{A,\lambda}\right)
\\
\leq
6\sigma\sqrt{2n\log\left(\frac{p}{\eta}\right)}  \sum_{j:\beta^{*}_{j} \neq 0} \xi_{j}^{\frac{1}{2}}(A)
 \left|\left(\hat{\beta}_{A,\lambda}\right)_{j}-\beta^{*}_{j}\right|
\\
\leq
6\sigma\sqrt{2n
 \sum_{j:\beta^{*}_{j} \neq 0} \left[\left(\hat{\beta}_{A,\lambda}\right)_{j}-\beta^{*}_{j}\right]^{2}
    \sum_{j:\beta^{*}_{j} \neq 0} \xi_{j}(A) \log\left(\frac{p}{\eta}\right) }
\\
\leq
6\sigma\sqrt{\frac{2}{ c(A'A)}
  \left(\beta^{*} - \hat{\beta}_{A,\lambda}\right) ' (A'A) \left( \beta^{*} - \hat{\beta}_{A,\lambda}\right)
  \sum_{j:\beta^{*}_{j} \neq 0} \xi_{j}(M) \log\left(\frac{p}{\eta}\right)},
\end{multline}
where we used Assumption~$H(A'A,3)$ for the last inequality. Then
\begin{equation}
\label{eq:prTropl}
 \left(\beta^{*} - \hat{\beta}_{A,\lambda}\right)' (A'A) \left( \beta^{*} - \hat{\beta}_{A,\lambda}\right)
\leq 72 \frac{\sigma^{2}}{c(A'A)}\log\left(\frac{p}{\eta}\right) \sum_{j:\beta^{*}_{j} \neq 0} \xi_{j}(A).
\end{equation}
A similar reasoning as in~\eqref{eq:profoBulf} leads to
\begin{multline*}
2 \sigma\sqrt{2n\log\left(\frac{p}{\eta}\right)}  \sum_{j=1}^{p} \xi_{j}^{\frac{1}{2}}(A)
 \left|\left(\hat{\beta}_{A,\lambda}\right)_{j}-\beta^{*}_{j}\right|
\\
\leq
8\sigma\sqrt{\frac{2}{ c(A'A)}
  \left(\beta^{*} - \hat{\beta}_{A,\lambda}\right) ' (A'A) \left( \beta^{*} - \hat{\beta}_{A,\lambda}\right)
  \sum_{j:\beta^{*}_{j} \neq 0} \xi_{j}(M) \log\left(\frac{p}{\eta}\right)}.
\end{multline*}
Finally, combine this last inequality with~\eqref{eq:prTropl} to obtain the desired bound for 
$
\left\|\Xi_{A}\left(\beta^{*} - \hat{\beta}_{A,\lambda}\right)\right\|_{1}.
$
This ends the proof.
\end{proof}

\begin{proof}[Proof of Theorem \ref{mainthm2}]
We have
\begin{multline}
\label{p11}
 (\tilde{\beta}_{A,\lambda}-\beta^{*})'(A'A)(\tilde{\beta}_{A,\lambda}-\beta^{*})
 = [\Xi_{A} (\tilde{\beta}_{A,\lambda}-\beta^{*})]'\Xi^{-1}_{A} (A'A)(\tilde{\beta}_{A,\lambda}-\beta^{*})
\\
 \leq \| \Xi_{A} (\tilde{\beta}_{A,\lambda}-\beta^{*})\|_{1}
 \|\Xi^{-1}_{A} (A'A)(\tilde{\beta}_{A,\lambda}-\beta^{*})\|_{\infty}
\\
\leq \| \Xi_{A} (\tilde{\beta}_{A,\lambda}-\beta^{*})\|_{1} \Biggl\{
 \|\Xi^{-1}_{A} (A'A)((\widetilde{X'X})^{-1}X'Y-\beta^{*})\|_{\infty}
\\
 + \|\Xi^{-1}_{A} (A'A)((\widetilde{X'X})^{-1}X'Y-\tilde{\beta}_{A,\lambda})\|_{\infty}
\Biggr\},
\end{multline}
by the constraint in the definition on $\tilde{\beta}_{A,\lambda}$ we have
$$ \|\Xi^{-1}_{A} (A'A)((\widetilde{X'X})^{-1}X'Y-\tilde{\beta}_{A,\lambda})\|_{\infty} \leq \lambda, $$
while Lemma \ref{usefullemma} implies that for $\lambda = 2\sigma\sqrt{2n\log(p/\eta)}$ we have
$$ \|\Xi^{-1}_{A} (A'A)((\widetilde{X'X})^{-1}X'Y-\beta^{*})\|_{\infty} \leq \frac{\lambda}{2} ,$$
with probability at least $1-\eta$; and so:
$$
(\tilde{\beta}_{A,\lambda}-\beta^{*})'(A'A)(\tilde{\beta}_{A,\lambda}-\beta^{*})
\leq \frac{3\lambda}{2} \| \Xi_{A} (\tilde{\beta}_{A,\lambda}-\beta^{*})\|_{1}.
$$
Moreover note that, by definition,
\begin{multline*}
0 \leq \|\Xi_{A} \beta^{*}\|_{1}-\|\Xi_{A} \tilde{\beta}_{A,\lambda}\|_{1}
\\
= \sum_{\beta^{*}_{j} \neq 0} \xi_{j}^{\frac{1}{2}}(A) \left|\beta^{*}_{j}\right|
        - \sum_{\beta^{*}_{j} \neq 0} \xi_{j}^{\frac{1}{2}}(A) \left|(\tilde{\beta}_{A,\lambda})_{j}\right|
        - \sum_{\beta^{*}_{j} = 0} \xi_{j}^{\frac{1}{2}}(A) \left|(\tilde{\beta}_{A,\lambda})_{j}\right|
\\
\leq \sum_{\beta^{*}_{j} \neq 0} \xi_{j}^{\frac{1}{2}}(A)
\left|\beta^{*}_{j}-(\tilde{\beta}_{A,\lambda})_{j}\right|
     - \sum_{\beta^{*}_{j} = 0} \xi_{j}^{\frac{1}{2}}(A)
\left|\beta^{*}_{j}-(\tilde{\beta}_{A,\lambda})_{j}\right|,
\end{multline*}
this implies that $\beta^{*}-(\tilde{\beta}_{A,\lambda})$ is an admissible vector in the
relation that defines Assumption $H(A'A,1)$.
Let us combine this result with Inequality~\eqref{p11}, we obtain
\begin{multline}
\label{p12}
(\tilde{\beta}_{A,\lambda}-\beta^{*})'(A'A)(\tilde{\beta}_{A,\lambda}-\beta^{*})
\leq
\frac{3\lambda}{2}
\| \Xi_{A} (\beta^{*} - \tilde{\beta}_{A,\lambda} )\|_{1}
\\
\leq 3\lambda \sum_{\beta^{*}_{j} \neq 0} \xi_{j}^{\frac{1}{2}}(A)
\left|\beta^{*}_{j}-(\tilde{\beta}_{A,\lambda})_{j}\right|
\\
\leq
3\lambda \sqrt{
       \left(\sum_{\beta^{*}_{j}\neq 0}\xi_{j}(A)\right)
       \left( \sum_{\beta^{*}_{j}\neq 0} \left|\beta^{*}_{j}-(\tilde{\beta}_{A,\lambda})_{j}\right|^{2}\right)}
\\
\leq
3\lambda  \left( \sum_{\beta^{*}_{j}\neq 0}\xi_{j}(A) \right)^{\frac{1}{2}}
       \sqrt{\frac{1}{n c(A'A)} (\tilde{\beta}_{A,\lambda}-\beta^{*})'(A'A)(\tilde{\beta}_{A,\lambda}-\beta^{*})}
.
\end{multline}
So we have,
$$
(\tilde{\beta}_{A,\lambda}-\beta^{*})'(A'A)(\tilde{\beta}_{A,\lambda}-\beta^{*})
\leq
9 \lambda^{2} \frac{1}{n c(A'A)}  \sum_{\beta^{*}_{j}\neq 0}\xi_{j}(A) ,
$$
and as a consequence, Inequality \eqref{p12} gives the upper bound on
$\| \Xi_{A} (\tilde{\beta}_{A,\lambda}-\beta^{*})\|_{1}$,
and this ends the proof.
\end{proof}

\subsection{Proof of Theorem \ref{mainthm3}}

\begin{proof}[Proof of Theorem \ref{mainthm3}]
The proof is almost the same as in the previous case. For the sake of simplicity, let
us write $\tilde{\beta}^{*}$ instead of $\tilde{\beta}^{*}_{\sqrt{n/m}Z,\lambda_{2}}$ and
the same for $\hat{\beta}^{*}$. We first give a look at the Dantzig Selector:
\begin{multline}
\label{p311}
 \frac{n}{m}
 \left(\tilde{\beta}^{*} -\beta^{*} \right)' Z'Z \left(\tilde{\beta}^{*} -\beta^{*} \right)
 \leq \left\|\tilde{\beta}^{*} - \beta^{*}\right\|_{1}
\left\|\frac{n}{m}  Z'Z \left(\tilde{\beta}^{*} -\beta^{*} \right) \right\|_{\infty}
\\
\leq
\left\|\tilde{\beta}^{*} -\beta^{*}\right\|_{1} \Biggl\{
\left\|\frac{n}{m} Z'  \left( Z \tilde{\beta}^{*} - \check{Y}_{\lambda_{1}} \right) \right\|_{\infty}
+ \left\|\frac{n}{m} Z'  \left( Z \beta^{*} - \check{Y}_{\lambda_{1}} \right)  \right\|_{\infty}
\Biggr\}
\\
\leq
\left\|\tilde{\beta}^{*} -\beta^{*}\right\|_{1} \Biggl\{
\left\|\frac{n}{m} Z' \left( Z \tilde{\beta}^{*} - \check{Y}_{\lambda_{1}} \right) \right\|_{\infty}
+
\left\|X'  \left( X \beta^{*} - Y \right)  \right\|_{\infty}
\\ +
\left\|X'  \left( X \tilde{\beta}_{X,\lambda_1} - Y \right)  \right\|_{\infty}
+
\left\|\left(\frac{n}{m} Z'Z  -X'X\right) \left( \beta^{*} - \tilde{\beta}_{X,\lambda_1} \right)  \right\|_{\infty}
\Biggr\}.
\end{multline}
By Lemma $\ref{usefullemma}$, for $\lambda_1 = 10^{-1} \sigma\sqrt{2n\log(p/\eta)}$ we have
$$ \|X'Y - X'X \beta^* \|_{\infty} \leq 10 \lambda_1 ,$$
with probability at least $1-\eta$. On the other hand, we have
$$ \| \beta^{*} - \tilde{\beta}_{X,\lambda_1} \|_{1} \leq
       \| \beta^{*} \|_{1} + \| \tilde{\beta}_{X,\lambda_1} \|_{1} \leq 2 \| \beta^{*} \|_{1}, $$
by definition of the Dantzig Selector. Then, let $u = (\beta^{*} - \tilde{\beta}_{X,\lambda_1} )/2 $ and use Inequality \eqref{hypoinf} for this specific $u$. This ensures that
\begin{equation}
\label{eq:vene}
\left\|\left(\frac{n}{m} Z'Z  -X'X\right) \left( \beta^{*} - \tilde{\beta}_{X,\lambda_1} \right)  \right\|_{\infty} \leq 2 \lambda_1.
\end{equation}
The definition of the Dantzig Selector also implies that
$$ \left\|X'  \left( X \tilde{\beta}_{X,\lambda_1 } - Y \right)  \right\|_{\infty} \leq \lambda_1 ,$$
and finally the definition of the estimator leads to
$$ \left\|\frac{n}{m} Z' \left( Z \tilde{\beta}^{*} - \check{Y}_{\lambda_{1}} \right) \right\|_{\infty}
\leq \lambda_2 = \lambda_1, $$
and as a consequence, Inequality \eqref{p311} becomes
$$
\frac{n}{m}
 \left(\tilde{\beta}^{*} -\beta^{*} \right)' Z'Z \left(\tilde{\beta}^{*} -\beta^{*} \right)
\leq
14 \lambda_1 \left\|\tilde{\beta}^{*} - \beta^{*}\right\|_{1}.
$$
Using the fact that $\|\tilde{\beta}^{*}\|_{1} \leq \|\beta^{*}\|_{1}$ gives
\begin{multline}
\label{p312}
\frac{n}{m}
 \left(\tilde{\beta}^{*} -\beta^{*} \right)' Z'Z \left(\tilde{\beta}^{*} -\beta^{*} \right)
\leq
14 \lambda_1 \left\|\tilde{\beta}^{*} - \beta^{*}\right\|_{1}
\leq 28 \lambda_1 \sum_{\beta^{*}_{j} \neq 0}
\left|\beta^{*}_{j}-(\tilde{\beta}^{*})_{j} \right|
\\
\leq 28 \lambda_1 \sqrt{\left|\left\{j:\beta^*_j \neq 0\right\}\right|
       \left( \sum_{\beta^{*}_{j}\neq 0} \left|\beta^{*}_{j}-(\tilde{\beta}^{*})_{j}\right|^{2}\right)}
\\
\leq
28 \lambda_1  \left|\left\{j:\beta^*_j \neq 0\right\}\right|^{\frac{1}{2}}
       \sqrt{\frac{1}{n c(n/m(Z'Z))} \frac{n}{m}
        \left(\tilde{\beta}^{*} -\beta^{*} \right)' Z'Z \left(\tilde{\beta}^{*} -\beta^{*} \right)}
.
\end{multline}
To establish the last inequality, we used Assumption $H'((n/m)Z'Z,1)$.
Then we have,
$$
\frac{n}{m}
 \left(\tilde{\beta}^{*} -\beta^{*} \right)' Z'Z \left(\tilde{\beta}^{*} -\beta^{*} \right)
\leq 28^2 \lambda_1^{2} \left|\left\{j:\beta^*_j \neq 0\right\}\right| \frac{1}{n c(n/m(Z'Z))},
$$
This inequality, combined with~\eqref{p312}, end the proof for the Dantzig Selector.

Now, let us deal with the LASSO case. The dual form of the definition of the estimator leads to
\begin{multline*}
-2\frac{n}{m}
\check{Y}_{\lambda_{1}} Z \hat{\beta}^{*} + \frac{n}{m}
(\hat{\beta}^{*})'Z'Z \hat{\beta}^{*} + 40 \lambda_2 \| \hat{\beta}^{*} \|_1
\\
\leq
-2\frac{n}{m}
\check{Y}_{\lambda_{1}} Z \beta^{*} + \frac{n}{m}
(\beta^{*})'Z'Z \beta^{*} +  40 \lambda_2  \| \beta^{*} \|_1
\end{multline*}
and so
\begin{multline*}
-2 \frac{n}{m}
\tilde{\beta}_{X,\lambda_1}  Z'Z \hat{\beta}^{*} + \frac{n}{m}
(\hat{\beta}^{*})'Z'Z \hat{\beta}^{*}
+ 40\lambda_2 \| \hat{\beta}^{*} \|_1
\\
\leq
-2 \frac{n}{m}
\tilde{\beta}_{X,\lambda_1}  Z'Z \beta^{*} + \frac{n}{m}
(\beta^{*})'Z'Z \beta^{*} + 40 \lambda_2 \| \beta^{*} \|_1
.
\end{multline*}
As a consequence,
\begin{multline*}
\frac{n}{m}
\left(\hat{\beta}^{*} -\beta^{*} \right)' Z'Z \left(\hat{\beta}^{*} -\beta^{*} \right)
\\
\leq
2 \frac{n}{m}
\left(\hat{\beta}^{*} -\beta^{*} \right)' Z'Z \left(\tilde{\beta}_{X,\lambda_1} -\beta^{*} \right)
+
40 \lambda_2 \left(\| \beta^{*} \|_1 -\| \hat{\beta}^{*} \|_1\right).
\end{multline*}
Now, we try to upper bound $ \left(\hat{\beta}^{*} -\beta^{*} \right)' Z'Z \left(\tilde{\beta}_{X,\lambda_1}
-\beta^{*} \right)$. We remark that
\begin{eqnarray*}
 & & \frac{n}{m} \left(\hat{\beta}^{*} -\beta^{*} \right)' Z'Z \left(\tilde{\beta}_{X,\lambda_1} -\beta^{*} \right)
\leq
\left\| \hat{\beta}^{*} -\beta^{*} \right\|_{1}  \left\| \frac{n}{m}(Z'Z)
       \left(\tilde{\beta}_{X,\lambda_1} -\beta^{*}\right) \right\|_{\infty}
\\
&
\leq
&
\left\| \hat{\beta}^{*} -\beta^{*} \right\|_{1}  \Biggl[
\left\| \left((\frac{n}{m}Z'Z-X'X\right)
       \left(\tilde{\beta}_{X,\lambda_1} -\beta^{*}\right) \right\|_{\infty}
\\
&&
+ \left\| X'X
       \left(\tilde{\beta}_{X,\lambda_1} -\beta^{*}\right) \right\|_{\infty}
\Biggr]
\leq 13 \lambda_1 \left\| \hat{\beta}^{*} -\beta^{*} \right\|_{1} ,
\end{eqnarray*}
where we used~\eqref{eq:vene} and the fact that
$$
\left\| X'X
       \left(\tilde{\beta}_{X,\lambda_1} -\beta^{*}\right) \right\|_{\infty} \leq \left\| X'
       \left(X\tilde{\beta}_{X,\lambda_1} - Y \right) \right\|_{\infty} + \left\| X'\varepsilon
        \right\|_{\infty} \leq \lambda_1 + 10 \lambda_1 = 11 \lambda_1.
$$
Then we have
\begin{multline*}
\frac{n}{m}
\left(\hat{\beta}^{*} -\beta^{*} \right)' Z'Z \left(\hat{\beta}^{*} -\beta^{*} \right)
\\
\leq
26 \lambda_1 \left\| \hat{\beta}^{*} -\beta^{*} \right\|_{1} +
40\lambda_2 \left(\| \beta^{*} \|_1 -\| \hat{\beta}^{*} \|_1\right),
\end{multline*}
and so
\begin{multline*}
\frac{n}{m}
\left(\hat{\beta}^{*} -\beta^{*} \right)' Z'Z \left(\hat{\beta}^{*} -\beta^{*} \right)
+14 \lambda_1 \left\| \hat{\beta}^{*} -\beta^{*} \right\|_{1}
\\
\leq
40 \lambda_1 \left( \left\| \hat{\beta}^{*} -\beta^{*} \right\|_{1}+
\| \beta^{*} \|_1 -\| \hat{\beta}^{*} \|_1\right).
\end{multline*}
Up to a multiplying constant, the rest of the proof of
Theorem~\ref{mainthm3} is the same as the last lines in the proof of Theorem~\ref{mainthm}. Then we omit it here.
\end{proof}

\subsection{Proof of Proposition \ref{propmatrix}}

\begin{proof}[Proof of Proposition \ref{propmatrix}]
First, let us remark that
\begin{multline*}
\left\|\left(X'X-\frac{n}{m}Z'Z\right) u \right\|_{\infty}
    = n \sup_{1\leq i \leq p} \sum_{j=1}^{p} u_{j} \left(\frac{X_{i}'X_{j}}{n}
                           -\frac{Z_{i}'Z_{j}}{m} \right)
\\
\leq n \left\| u \right\|_{1} \sup_{1\leq i,j \leq p}   \left|\frac{X_{i}'X_{j}}{n}
                           -\frac{Z_{i}'Z_{j}}{m} \right|.
\end{multline*}
Now, using the "exchangeable-distribution inequality" in \cite{manuscrit} we obtain, for a
given pair $(i,j)$, for any $\tau>0$, with probability at least $1-\eta$,
\begin{multline*}
\frac{X_{i}'X_{j}}{n}-\frac{Z_{i}'Z_{j}}{m}
\leq \frac{\tau k^2}{2n(k+1)^2} \left(\frac{1}{m}\sum_{k=1}^{m} X_{i,k}^{2} X_{j,k}^{2} \right)
       + \frac{\log \frac{1}{\eta}}{\tau}
\\
\leq \frac{\tau k^2 \kappa^2 }{2n(k+1)^2}
       + \frac{\log \frac{1}{\eta}}{\tau}
= \frac{\kappa k}{k-1}\sqrt{\frac{2\log\frac{1}{\eta}}{n}},
\end{multline*}
for $\tau=(\log(1/\eta)(k-1)2n/k\kappa^2)^{1/2}$
and so, by a union bound argument, with probability at least $1-\eta$,
for any pair $(i,j)$,
$$
\left|\frac{X_{i}'X_{j}}{n}-\frac{Z_{i}'Z_{j}}{m}\right|
\leq
\frac{\kappa k}{k-1}\sqrt{\frac{2\log\frac{2p^{2}}{\eta}}{n}}
\leq
\frac{2 \kappa k}{k-1}\sqrt{\frac{2\log\frac{p}{\eta}}{n}},
$$
(where we used $p\geq 2$).
\end{proof}


\bibliographystyle{alpha}
\bibliography{GenLasDan}

\newcommand{\etalchar}[1]{$^{#1}$}
\begin{thebibliography}{MVdGB08}

\bibitem[AH08]{L1MOH}
P.~Alquier and M.~Hebiri.
\newblock Generalization of l1 constraint for high-dimensional regression
  problems.
\newblock Preprint Laboratoire de Probabilit\'es et Mod\`eles Al\'eatoires (n.
  1253), arXiv:0811.0072, 2008.

\bibitem[Aka73]{aic}
H.~Akaike.
\newblock Information theory and an extension of the maximum likelihood
  principle.
\newblock In B.~N. Petrov and F.~Csaki, editors, {\em 2nd International
  Symposium on Information Theory}, pages 267--281. Budapest: Akademia Kiado,
  1973.

\bibitem[Alq08]{CSEL}
P.~Alquier.
\newblock Lasso, iterative feature selection and the correlation selector:
  Oracle inequalities and numerical performances.
\newblock {\em Electron. J. Stat.}, pages 1129--1152, 2008.

\bibitem[Bac08]{BachGpLasso}
F.~Bach.
\newblock Consistency of the group lasso and multiple kernel learning.
\newblock {\em J. Mach. Learn. Res.}, 9:1179--1225, 2008.

\bibitem[BRT07]{Lasso3}
P.~Bickel, Y.~Ritov, and A.~Tsybakov.
\newblock Simultaneous analysis of lasso and dantzig selector.
\newblock Submitted to the Ann. Statist., 2007.

\bibitem[BTW07]{BTWAggSOI}
F.~Bunea, A.~Tsybakov, and M.~Wegkamp.
\newblock Aggregation for {G}aussian regression.
\newblock {\em Ann. Statist.}, 35(4):1674--1697, 2007.

\bibitem[Bun08]{Bunea_consist}
F.~Bunea.
\newblock {\em Consistent selection via the Lasso for high dimensional
  approximating regression models}, volume~3.
\newblock IMS Collections, 2008.

\bibitem[Cat07]{manuscrit}
O.~Catoni.
\newblock {\em PAC-Bayesian Supervised Classification (The Thermodynamics of
  Statistical Learning)}, volume~56 of {\em Lecture Notes-Monograph Series}.
\newblock IMS, 2007.

\bibitem[CH08]{ChriMo7GpLass}
C.~Chesneau and M.~Hebiri.
\newblock Some theoretical results on the grouped variables lasso.
\newblock {\em Mathematical Methods of Statistics}, 17(4):317--326, 2008.

\bibitem[CSZ06]{semi-sup}
O.~Chapelle, B.~Sch\"olkopf, and A.~Zien.
\newblock {\em Semi-supervised learning}.
\newblock MIT Press, Cambridge, MA, 2006.

\bibitem[CT07]{Dantzig}
E.~Cand\`es and T.~Tao.
\newblock The dantzig selector: statistical estimation when $p$ is much larger
  than $n$.
\newblock {\em Ann. Statist.}, 35, 2007.

\bibitem[DT07]{ArnakTsyb}
A.~Dalalyan and A.B. Tsybakov.
\newblock Aggregation by exponential weighting and sharp oracle inequalities.
\newblock {\em COLT 2007 Proceedings. Lecture Notes in Computer Science 4539
  Springer}, pages 97--111, 2007.

\bibitem[EHJT04]{Efron-LARS}
B.~Efron, T.~Hastie, I.~Johnstone, and R.~Tibshirani.
\newblock Least angle regression.
\newblock {\em Ann. Statist.}, 32(2):407--499, 2004.
\newblock With discussion, and a rejoinder by the authors.

\bibitem[FHHT07]{PCO}
J.~Friedman, T.~Hastie, H.~H\"ofling, and R.~Tibshirani.
\newblock Pathwise coordinate optimization.
\newblock {\em Ann. Appl. Statist.}, 1(2):302--332, 2007.

\bibitem[FL01]{FanLiScad}
J.~Fan and R.~Li.
\newblock Variable selection via nonconcave penalized likelihood and its oracle
  properties.
\newblock {\em J. Amer. Statist. Assoc.}, 96(456):1348--1360, 2001.

\bibitem[HCB08]{Barron2}
C.~Huang, G.~L.~H. Cheang, and A.~Barron.
\newblock Risk of penalized least squares, greedy selection and l1 penalization
  for flexible function libraries.
\newblock preprint, 2008.

\bibitem[Heb08]{Mo7SLasso}
M.~Hebiri.
\newblock Regularization with the smooth-lasso procedure.
\newblock Preprint LPMA, 2008.

\bibitem[JRL09]{DASSO}
G.~James, P.~Radchenko, and J.~Lv.
\newblock Dasso: Connections between the dantzig selector and lasso.
\newblock {\em JRSS (B)}, 71:127--142, 2009.

\bibitem[KKL{\etalchar{+}}07]{interior}
S.~J. Kim, K.~Koh, M.~Lustig, S.~Boyd, and D.~Gorinevsky.
\newblock An interior-point method for large-scale l1-regularized least
  squares.
\newblock {\em IEEE Journal of Selected Topics in Signal Processing},
  1(4):606--617, 2007.

\bibitem[Kol07]{KoltchDant}
V.~Koltchinskii.
\newblock Dantzig selector and sparsity oracle inequalities.
\newblock Preprint, 2007.

\bibitem[Kol09]{Koltchl1plus}
V.~Koltchinskii.
\newblock Sparse recovery in convex hulls via entropy penalization.
\newblock {\em Annals of Statistics}, 37(3):1332--1359, 2009.

\bibitem[Lou08]{KarimNormSup}
K.~Lounici.
\newblock Sup-norm convergence rate and sign concentration property of {L}asso
  and {D}antzig estimators.
\newblock {\em Electron. J. Stat.}, 2:90--102, 2008.

\bibitem[MB06]{MeinshBulhmConsistLasso}
N.~Meinshausen and P.~B{\"u}hlmann.
\newblock High-dimensional graphs and variable selection with the lasso.
\newblock {\em Ann. Statist.}, 34(3):1436--1462, 2006.

\bibitem[MVdGB08]{VanGpLass}
L.~Meier, S.~Van~de Geer, and P.~B{\"u}hlmann.
\newblock High-dimensional additive modeling.
\newblock To appear in the Annals of Statistics, 2008.

\bibitem[MY09]{MeinYuSelect}
N.~Meinshausen and B.~Yu.
\newblock Lasso-type recovery of sparse representations for high-dimensional
  data.
\newblock {\em Ann. Statist.}, 37(1):246--270, 2009.

\bibitem[OPT00]{DualLasso}
M.~Osborne, B.~Presnell, and B.~Turlach.
\newblock On the {LASSO} and its dual.
\newblock {\em J. Comput. Graph. Statist.}, 9(2):319--337, 2000.

\bibitem[Sch78]{bic}
G.~Schwarz.
\newblock Estimating the dimension of a model.
\newblock {\em The Annals of Statistics}, 6:461--464, 1978.

\bibitem[Tib96]{Tibshirani-LASSO}
R.~Tibshirani.
\newblock Regression shrinkage and selection via the lasso.
\newblock {\em J. Roy. Statist. Soc. Ser. B}, 58(1):267--288, 1996.

\bibitem[Vap98]{Vapnik}
V.~Vapnik.
\newblock {\em The Nature of Statistical Learning Theory}.
\newblock Springer-Verlag, 1998.

\bibitem[vdG08]{VandeGeerSparseLasso}
S.~van~de Geer.
\newblock High-dimensional generalized linear models and the lasso.
\newblock {\em Ann. Statist.}, 36(2):614--645, 2008.

\bibitem[Wai06]{WainSelection}
M.~Wainwright.
\newblock Sharp thresholds for noisy and high-dimensional recovery of sparsity
  using l1-constrained quadratic programming.
\newblock Technical report n. 709, Department of Statistics, UC Berkeley, 2006.

\bibitem[YL07]{garrotte}
M.~Yuan and Y.~Lin.
\newblock On the non-negative garrotte estimator.
\newblock {\em J. R. S. S. (B)}, 69(2):143--161, 2007.

\bibitem[Zou06]{AdapLassoZou}
H.~Zou.
\newblock The adaptive lasso and its oracle properties.
\newblock {\em J. Amer. Statist. Assoc.}, 101(476):1418--1429, 2006.

\bibitem[ZY06]{BiYuConsistLasso}
P.~Zhao and B.~Yu.
\newblock On model selection consistency of {L}asso.
\newblock {\em J. Mach. Learn. Res.}, 7:2541--2563, 2006.

\end{thebibliography}

\end{document}